\documentclass[]{article}
\usepackage{amsmath,amsthm,amssymb,amsfonts,mathtools,enumerate,enumitem,authblk,setspace}
\usepackage[round]{natbib}

\theoremstyle{theorem}
\newtheorem{theorem}{Theorem}

\newtheorem{proposition}{Proposition}

\theoremstyle{definition}
\newtheorem{example}{Example}

\theoremstyle{remark}
\newtheorem{remark}{Remark}

\usepackage{hyperref}
\hypersetup{
    colorlinks,
    citecolor=black,
    filecolor=red,
    linkcolor=black,
    urlcolor=red
}
\usepackage{indentfirst}
\numberwithin{equation}{section}
\allowdisplaybreaks

\AtBeginDocument{}


\setlist{nolistsep}

\setenumerate[1]{label=\Roman*.}
\setenumerate[2]{label=\Alph*.}
\setenumerate[3]{label=\roman*.}
\setenumerate[4]{label=\alph*.}

\begin{document}
\date{}
\title{Solving Poisson's Equation: Existence, Uniqueness, Martingale Structure, and CLT}
\author[1,2]{Peter W. Glynn}
\author[2]{Alex Infanger}
\affil[1]{Department of Management Science \& Engineering, Stanford University.}
\affil[2]{Institute for Computational \& Mathematical Engineering, Stanford University.}
\maketitle

\begin{abstract} \noindent 
The solution of Poisson’s equation plays a key role in constructing the martingale through which sums of Markov correlated random variables can be analyzed. In this paper, we study two different representations for the solution in countable state space, one based on regenerative structure and the other based on an infinite sum of expectations. We also consider integrability and related uniqueness issues associated with solutions to Poisson’s equation, and provide verifiable Lyapunov conditions to support our theory. Our key results include a central limit theorem and law of the iterated logarithm for Markov dependent sums, under Lyapunov conditions weaker than have previously appeared in the literature.
\end{abstract}

\section{Introduction}
Let $X=(X_n:n\geq 0)$ be an irreducible positive recurrent Markov chain taking values in a finite or countably infinite state space $S$. We let $P=(P(x,y):x,y\in S)$ be the one-step transition matrix of $X$, and let $\pi=(\pi(x):x\in S)$ be its associated (unique) stationary distribution (encoded as a row vector). Given a function $f:S\rightarrow \mathbb R$ (encoded as a column vector), we say that a function $g$ is a solution of \emph{Poisson's equation} (for $f$) if
\begin{align}
(P-I)g = -f.\label{eq11}
\end{align}
Poisson's equation is fundamental to the analysis of the additive functional $S_n(f)\overset{\Delta}{=}\sum_{i=0}^{n-1}f(X_i)$, since, in the presence of integrability,
\begin{align}
g(X_n) + \sum_{i=0}^{n-1}f(X_i)\label{eq12}
\end{align}
should then be a martingale adapted to the filtration $(\mathcal{F}_n:n\geq 0)$, where $\mathcal{F}_n=\sigma(X_j:j\leq n)$. The martingale representation \eqref{eq12} can then be used to advantage in computing expected values (e.g. via optional sampling), and in deriving the law of large numbers, central limit theorem, and law of the iterated logarithm for $S_n(f)$; see, for example, \citet{maigretTheoremeLimiteCentrale1978} and \citet{kurtzCentralLimitTheorem1981}. It is also fundamental to obtaining bounds on the $k$'th moment $E_x |S_n(f)|^k$ (where $E_x(\cdot)$ is the expectation corresponding to the probability $P_x(\cdot)=P(\cdot|X_0=x)$) via the Burkholder-Davis-Gundy inequality; see \citet{hallMartingaleLimitTheory2014}. It is worth noting that \eqref{eq11} also arises implicitly within the optimality equation for average reward/average case stochastic control problems, since the value function under the optimal policy must satisfy \eqref{eq11}; see \citet{rossIntroductionStochasticDynamic2014}.

Given the central importance of \eqref{eq11} within applied probability, this paper is intended to clarify the question of existence and uniqueness of solutions to \eqref{eq11} when $|S|=\infty$. Assuming $g\in L^p(\pi)\overset{\Delta}{=}\{h: \sum_{x\in S}^{}\pi(x)|h(x)|^p<\infty\}$ for $p\geq 1$, \eqref{eq11} implies that $f\in L^1(\pi)$ and $\pi f=0$. If $f\in L_0^1(\pi)=\{h:L^1(\pi): \pi h=0\}$, $|S|<\infty$, and $X$ is aperiodic, it turns out that a solution $g$ to \eqref{eq11} is given either by
\begin{align*}
g_1(x) = E_x\sum_{n=0}^{\tau(z)-1} f(X_n)
\end{align*}
or by 
\begin{align*}
g_2(x) = \sum_{n=0}^{\infty}E_xf(X_n),
\end{align*}
where $\tau(z)=\inf\{n\geq 1: X_n=z\}$ for $z\in S$. In this finite state setting, both $g_1$ and $g_2$ are always finite-valued solutions to \eqref{eq11} and $g_1(x)-g_2(x)$ is constant as a function of $x\in S$.

When $|S|=\infty$, both existence of solutions of \eqref{eq11} and the corresponding uniqueness issues are more subtle. In particular, this paper shows that
\begin{enumerate}[label=\alph*)]
\item While $g_1$ is always a solution to \eqref{eq11} when $f\in L^1(\pi)$, $g_2$ can sometimes fail to be well-defined (see Theorem \ref{thm1} and Example \ref{example2});
\item The set of solutions to \eqref{eq11} can be infinite-dimensional (see Example \ref{example1});
\item When $g_1$ is used in \eqref{eq12}, \eqref{eq12} is guaranteed to be integrable and hence a martingale (Theorem \ref{thm1});
\item The set of functions $g$ for which \eqref{eq12} is a martingale can be infinite-dimensional (Example \ref{example1}), but we show that it becomes one-dimensional when we require a certain uniform integrability property (Theorem \ref{thm6});
\item An approximation to $E_xS_n(f)$ can be developed when $g_1\in L^1(\pi)$ (Theorem \ref{thm3}), and a Lyapunov criterion is provided (Theorem \ref{thm5});
\item The function $g_2$ is well-defined under an extra condition (Theorem \ref{thm4}) that can be verified through an associated Lyapunov criterion (Theorem \ref{thm5});
\item Lyapunov criteria, weaker than those previously developed in the literature, are provided that guarantee the central limit theorem (CLT) and law of the iterated logarithm (LIL) for $S_n(f)$ (Theorems \ref{thm7} and \ref{thm8}).
\end{enumerate}
As we shall see in what follows, the conditions we will develop are close to being minimal. Section \ref{sec2} contains the key results on Poisson's equation, while Section \ref{sec3} discusses two examples that illustrate mathematical subtleties associated with \eqref{eq11}. Section \ref{sec4} describes the CLT and LIL for $S_n(f)$. As an application of our theory, \citet{glynnSolutionsPoissonEquation} uses the characterization $g_1$ of the solution to establish that solutions of Poisson's equation are always monotone when the chain is stochastically monotone and $f$ is monotone, without demanding the summability implicit in $g_2$.

\section{A Regenerative Representation for the Solution and Related Lyapunov Conditions}\label{sec2}
As noted in the Introduction, we assume throughout this paper that $X$ is irreducible and positive recurrent. This implies that
\begin{align}
E_z\tau(z)<\infty\label{eq21}
\end{align}
for each $z\in S$. This, in turn, implies that $E_x \tau(z)<\infty$ for $x,z\in S$, since
\begin{align}
\infty> E_z \tau(z)&\geq E_z \tau(z)I(\tau(x)<\tau(z))\nonumber\\
&\geq P_z\left(\tau(x)<\tau(z)\right)E_x\tau(z).\label{eq22}
\end{align}
Of course, $P_z(\tau(x)<\tau(z))>0$ (because $x$ can not be reached from $z$ if $P_x(\tau(x)<\tau(z))=0$), thereby establishing that $E_x \tau(z)<\infty$.

Similarly, if $f\in L^1(\pi)$, the regenerative structure of $X$ implies that for each $z\in S$, 
\begin{align*}
\sum_{x\in S}^{}\pi(x)|f(x)| = \frac{E_z\sum_{j=0}^{\tau(z)-1}|f(X_j)|}{E_z\tau(z)},
\end{align*}
and hence
\begin{align*}
E_z \sum_{j=0}^{\tau(z)-1}|f(X_j)|<\infty.
\end{align*}
Similarly as with $E_x\tau(z)<\infty$, this implies that
\begin{align}
E_x\sum_{j=0}^{\tau(z)-1} |f(X_j)|<\infty\label{eq23}
\end{align}
for $x,z\in S$.

We conclude that if $f\in L^1(\pi)$, then for each $x,z\in S$,
\begin{align*}
g_z(x) \overset{\Delta}{=} E_x\sum_{j=0}^{\tau(z)-1}f_c(X_j)
\end{align*}
is finite-valued, where $f_c(x)=f(x)-\pi f$ for $x\in S$. Our first result asserts that $g_z=(g_z(x):x\in S)$ is a solution of Poisson's equation for $f_c$, namely it satisfies
\begin{align}
(P-I)g = -f_c.\label{eq24}
\end{align}
It also establishes that for any $y,x\in S$, $g_z$ and $g_y$ differ only by an additive constant. Let $e=(e(x):x\in S)$ be the (constant) function for which $e(x)=1$ for $x\in S$, and put $a\land b\overset{\Delta}{=}\min(a,b)$ for $a,b\in \mathbb{R}$. 

\begin{theorem} \label{thm1} Suppose $f\in L^1(\pi)$. Then, 
\begin{enumerate}[label=\alph*)]
\item For each $z\in S$,  $g_z$ is a finite-valued function that satisfies \eqref{eq24};
\item For each $y,z\in S$, $g_z-g_y=g_z(y)e$.
\end{enumerate}
\end{theorem}

\begin{proof} We have already established that $g_z$ is finite-valued. For $x\in S$, conditioning on $X_1$ shows that
\begin{align}
g_z(x) &= E_x\sum_{j=0}^{\tau(z)-1}f_c(X_j)\nonumber\\
&= f_c(x) + \sum_{y\neq z}^{} P(x,y)g_z(y).\label{eq25}
\end{align}
But $\pi f_c=0$, and the regenerative structure of $X$ implies that
\begin{align*}
\pi f_c = \frac{E_z \sum_{j=0}^{\tau(z)-1}f_c(X_j)}{E_z\tau(z)},
\end{align*}
so that $g_z(z)=0$. Hence \eqref{eq25} yields the identity
\begin{align*}
g_z(x)=f_c(x) + \sum_{y\in S}^{}P(x,y)g_z(y),
\end{align*}
which is $a)$. As for $b)$, note that
\begin{align}
g_z(x) = E_x\smashoperator{\sum_{j=0}^{(\tau(y)\land \tau(z))-1}}f_c(X_j) + P_x(\tau(y)<\tau(z)) g_z(y).\label{eq26}
\end{align}
Since the identity also holds for $g_y(x)$ (with $y$ playing the role of $z$), we find that
\begin{align}
g_z(x) - g_y(x) = g_z(y) - P_x(\tau(z)<\tau(y))(g_y(z)+g_z(y))\label{eq27}
\end{align}
for $x\in S$. If we set $x=y$ in \eqref{eq27} and note that irreducibility implies that $P_z(\tau(z)<\tau(y))>0,$ we find that $g_y(z)-g_z(y)=0$, yielding b). 
\end{proof}
\begin{remark} Theorem \ref{thm1} is a special case of Theorem 1 and 2 of \citet{dermanSolutionCountableSystem1967}; see also \citet{schaufelePotentialTheoreticProof1967}. Our proof is easier because we assume irreducibility.
\end{remark}
We say that $(M_n:n\geq 0)$ is a $P_x-$\emph{martingale} adapted to $(\mathcal{F}_n:n\geq 0)$ if $(M_n:n\geq 0)$ is adapted to $(\mathcal{F}_n:n\geq 0)$, $E_x|M_n|<\infty$, and $E_x[M_{n+1}|\mathcal{F}_n]=M_n$ a.s. for $n\geq 0$.

\begin{theorem} \label{thm2} Suppose $f\in L^1(\pi)$. Then, for all $x,z\in S$, 
\begin{align}
g_z(X_n) + \sum_{j=0}^{n-1}f_c(X_j)\label{eq28}
\end{align}
is a $P_x$-martingale adapted to $(\mathcal{F}_n:n\geq 0)$. 
\end{theorem}
\begin{proof} The only non-trivial issue that needs to be addressed is the $P_x$-integrability of the martingale. Note that for $n\geq 0$, because $f\in L^1(\pi)$,
\begin{align}
\pi(x)E_x|f_c(X_n)|\leq \sum_{y\in S}^{}\pi(y)E_y|f_c(X_n)| = \sum_{y\in S}^{}\pi(y)|f_c(y)|<\infty
\end{align}
due to stationarity, and hence $f_c(X_n)$ is $P_x$-integrable for $n\geq 0$.

As for the integrability of $g_z(X_n)$, we observe that for $x\neq z$, 
\begin{align}
a_n(x)\overset{\Delta}{=}E_x|g_z(X_n)| = E_x|E_x\sum_{j=n}^{\tau(z)-1}f_c(X_j)|I(\tau(z)>n)+ \sum_{j=1}^{n}P_x(\tau(z)=j) a_{n-j}(z).\label{eq211}
\end{align}
As for $(a_n(z):n\geq 0)$, it satisfies the renewal equation 
\begin{align*}
a_n(z) = E_z|E_z\sum_{j=n}^{\tau(z)-1}f_c(X_j)|I(\tau(z)>n) + \sum_{j=1}^{n} P_z(\tau(z)=j)a_{n-j}(z)
\end{align*}
for $n\geq 0$, the solution of which is 
\begin{align}
a_n(z) = \sum_{j=0}^{n}P_z(X_j=z)b_{n-j}(z),\label{eq212}
\end{align}
where
\begin{align*}
b_n(x)&\leq E_x\sum_{j=n}^{\tau(z)-1}|f_c(X_j)|I(\tau(z)>n)\\
&\leq E_x\sum_{j=0}^{\tau(z)-1}|f_c(X_j)|<\infty
\end{align*}
for each $x\in S$, on account of \eqref{eq23}; see Chapter 13 of \citet{fellerIntroductionProbabilityTheory1968} for a discussion of discrete-time renewal theory. Hence, $(b_n(x):n\geq 0)$ is a finite-valued sequence for each $x\in S$, so that \eqref{eq211} and \eqref{eq212} imply the finiteness of $(a_n(x):n\geq 0)$ for $x\in S$, yielding the $P_x$-integrability of $g_z(X_n)$. 
\end{proof}

\begin{remark} Our result shows that whenever $f\in L^1(\pi)$, \eqref{eq28} is a $P_x$-martingale. Discussion of the martingale \eqref{eq28} appears elsewhere (e.g. in \citet{kurtzCentralLimitTheorem1981} and \cite{makowskiPoissonEquationCountable2002}), but without a simple sufficient condition for integrability.
\end{remark}

\begin{remark} For $\mu=(\mu(x):x\in S)$ a probability on $S$, let $P_\mu(\cdot)\overset{\Delta}{=}\sum_{x\in S}^{}\mu(x)P_x(\cdot)$ be the probability on the path-space of $X$ under which $X$ has initial distribution $\mu$, and let $E_\mu (\cdot)$ be its corresponding expectation operator. Then \eqref{eq28} is guaranteed to be a $P_\mu$-martingale provided that $\mu$ is finitely supported. However \eqref{eq28} is, in general, not integrable for infinitely supported $\mu$ (e.g. $\mu=\pi$).
\end{remark}
As a consequence of Theorem \ref{thm2},
\begin{align}
\sum_{j=0}^{n-1}E_xf(X_j) = n\pi f + g_z(x) - E_xg_z(X_n)\label{eq213}
\end{align}
for $x,z\in S$. To further simplify \eqref{eq213}, we can apply the following result.

\begin{theorem} \label{thm3} Suppose that $f\in L^1(\pi)$ and that there exists $z\in S$ for which 
\begin{align}
E_z \sum_{j=0}^{\tau(z)-1}j|f_c(X_j)|<\infty.\label{eq214}
\end{align}
Then, $g_y\in L^1(\pi)$ for all $y\in S$, and $\eqref{eq28}$ is a $P_\pi$-martingale adapted to $(\mathcal{F}_n:n\geq 0)$. Furthermore, if $X$ is aperiodic, then
\begin{align}
E_x g_z(X_n)\rightarrow \pi g_z=\frac{E_z\sum_{j=0}^{\tau(z)-1}(j+1)f_c(X_j)}{E_z\tau(z)}\label{eq215}
\end{align}
as $n\rightarrow\infty$ for $x\in S$.
\end{theorem}

\begin{proof} As in the argument leading to \eqref{eq211} and \eqref{eq212}, we find that 
\begin{align}
\tilde a_n(x) = E_x g_z(X_n) = \tilde b_n(x) + \sum_{k=1}^{n}P_x(\tau(z)=k)\sum_{j=0}^{n-k}\tilde b_j(z)P_z(X_{n-k-j}=z)\label{eq216}
\end{align}
where
\begin{align*}
\tilde b_n(x) = E_x\sum_{j=n}^{\tau(z)-1}f_c(X_j)I(\tau(z)>n).
\end{align*}
As a consequence of the aperiodicity, $P_z(X_n=z)\rightarrow \pi(z)$ as $n\rightarrow\infty$. Also,
\begin{align}
\sum_{n=0}^{\infty}|\tilde b_n(z)|&\leq E_z\sum_{n=0}^{\infty}\sum_{j=n}^{\tau(z)-1}|f_c(X_j)| I(\tau(z)>n)\nonumber\\
&=E_z \sum_{n=0}^{\tau(z)-1}\sum_{j=n}^{\tau(z)-1}|f_c(X_j)|\nonumber\\
&= E_z \sum_{j=0}^{\tau(z)-1} |f_c(X_j)| \sum_{n=0}^{j}1\nonumber\\
&=E_z \sum_{j=0}^{\tau(z)-1} (j+1)|f_c(X_j)|<\infty.\label{eq217}
\end{align}
With \eqref{eq217} in hand, we can apply the Dominated Convergence Theorem to \eqref{eq216}, thereby yielding \eqref{eq215}.

Furthermore, the regenerative structure of $X$ implies that if $\beta_n(z)=\inf\{j>n: X_j=z\}$, then 
\begin{align}
E_z\tau(z)E_\pi|g_z(X_0)| &= E_z \sum_{n=0}^{\tau(z)-1}|g_z(X_n)|\nonumber\\
&= E_z \sum_{n=0}^{\infty}|E_z \sum_{j=n}^{\beta_n(z)-1}f_c(X_j)|I(\tau(z)>n)\nonumber\\
&\leq E_z\sum_{n=0}^{\infty} \sum_{j=n}^{\beta_n(z)-1}|f_c(X_j)| I(\tau(z)>n)\nonumber\\
&= E_z \sum_{n=0}^{\infty}\sum_{j=n}^{\tau(z)-1}|f_c(X_j)|I(\tau(z)>n)\nonumber\\
&= E_z \sum_{n=0}^{\tau(z)-1}\sum_{j=n}^{\tau(z)-1}|f_c(X_j)|,\label{eq218-new}
\end{align}
which equals the right-hand side of \eqref{eq217}. It follows that $g_z\in L^1(\pi)$. But $g_y$ is equal to $g_z$ (up to an additive constant), so $g_y\in L^1(\pi)$ for $y\in S$. Obviously, if $g_y\in L^1(\pi)$, then $g_y(X_n)+\sum_{j=1}^{n-1}f_c(X_j)$ is a $P_\pi$-martingale, concluding the argument.

The expression for $\pi g_z$ follows from the same argument as that leading to \eqref{eq218-new}. 
\end{proof}

Under the conditions of Theorem \ref{thm3}, we see that
\begin{align}
E_x \sum_{j=0}^{n-1}f(X_j) = n\pi f + g_z(x)- \pi g_z + o(1) \label{eq218}
\end{align}
as $n\rightarrow\infty$, where $o(1)$ represents a sequence $(c_n:n\geq 0)$ with the property that $c_n\rightarrow 0$ as $n\rightarrow\infty$. Given the expression for $\pi g_z$, it is evident that Theorem \ref{thm3}'s hypotheses are close to necessary.

The next result shows that \eqref{eq214} is a ``solidarity property'', in the sense that if it holds for one $z\in S$, then it holds for each $y\in S$.

\begin{proposition} \label{prop1} Suppose $f\in L^1(\pi)$ and \eqref{eq214} is valid for some $z\in S$. Then, \eqref{eq214} holds for all $z\in S$.
\end{proposition}

\begin{proof} We note that the regenerative structure of $X$ implies that
\begin{align*}
\infty>E_z \sum_{j=0}^{\tau(z)-1}(j+1) |f_c(X_j)| &= E_z\sum_{k=0}^{\tau(z)-1}\sum_{j=k}^{\tau(z)-1}|f_c(X_j)|\\
&=E_z\sum_{\ell=0}^{\tau(z)-1}k_z(X_\ell) = E_z\tau(z)\cdot \pi k_z
\end{align*}
where $k_z(x)=E_x\sum_{j=0}^{\tau(z)-1} |f_c(X_j)|$ for $x\in S$. If $y\neq z$, relation \eqref{eq23} implies that $E_y\sum_{\ell=0}^{\tau(y)-1}|k_z(X_\ell)|<\infty$. Also,
\begin{align*}
E_y \sum_{j=0}^{\tau(y)-1}(j+1)|f_c(X_j)|&= E_y\sum_{j=0}^{\tau(y)-1}\sum_{k=j}^{\tau(y)-1}|f_c(X_k)|\\
&=E_y\sum_{j=0}^{\tau(y)-1}I(\beta_j(z)\geq\tau(y))\sum_{k=j}^{\tau(y)-1}|f_c(X_k)|\\
& \qquad + E_y \sum_{j=0}^{\tau(y)-1}I(\beta_j(z)<\tau(y))\sum_{k=j}^{\beta_j(z)-1}|f_c(X_k)|\\
& \qquad + E_y \sum_{j=0}^{\tau(y)-1}I(\beta_j(z)<\tau(y))\sum_{k=\beta_j(z)}^{\tau(y)-1}|f_c(X_k)|\\
&\leq E_y\sum_{j=0}^{\tau(y)-1}\sum_{k=j}^{\beta_j(z)-1}|f_c(X_k)|\\
&\qquad + E_y\sum_{j=0}^{\tau(y)-1}I(\beta_j(z)<\tau(y))E_y[\sum_{k=\beta_j(z)}^{\tau(y)-1}|f_c(X_k)|\ |X_0,...,X_{\beta_j(z)}]\\
&= E_y\sum_{j=0}^{\tau(y)-1}k_z(X_j) + E_y\sum_{j=0}^{\tau(y)-1}I(\beta_j(z)<\tau(y))E_z\sum_{j=0}^{\tau(y)-1}|f_c(X_k)|\\
&\leq E_y \sum_{j=0}^{\tau(y)-1}k_z(X_j) + E_y \tau(y)\cdot E_z\sum_{j=0}^{\tau(y)-1}|f_c(X_k)|<\infty,
\end{align*}
where we have used \eqref{eq23} in the last line, thereby proving the result. 
\end{proof}

Our final result in this section on the behavior of $\sum_{j=0}^{n}E_xf_c(X_j)$ concerns the development of a regenerative criterion that ensures the validity of 
\begin{align}
\sum_{j=0}^{\infty}|E_xf_c(X_j)|<\infty\label{eq219}
\end{align}
for $x\in S$.

\begin{theorem}\label{thm4} Suppose that $f\in L^1(\pi)$ and that there exists $z\in S$ for which \eqref{eq214} holds and $E_z\tau(z)^2<\infty$. If $X$ is aperiodic, then \eqref{eq219} holds for all $x\in S$. 
\end{theorem}
 \begin{proof} We first note that $E_z\tau(z)^2<\infty$ is a solidarity property, in the sense that if it holds for one $z\in S$, then it holds for all $z\in S$; see p.84 of \citet{chungMarkovChains1967}. As a consequence of Proposition \ref{prop1}, we may assume \eqref{eq214} and $E_z\tau(z)^2<\infty$ hold for $z=x$.

 Then, if $a_n^*=E_xf_c(X_n)$ it follows that it satisfies the renewal equation
 \begin{align*}
 a_n^* = \sum_{j=1}^{n}P_x(\tau(x)=j)a^*_{n-j} + b_n^*,
 \end{align*}
 where $b_n^*=E_x f_c(X_n)I(\tau(x)>n)$. So,
 \begin{align*}
 a_n^* = \sum_{j=0}^{n}P_x(X_j=x) b_{n-j}^*.
 \end{align*}
 Recalling that $0=\pi f_c=\sum_{j=0}^{\infty}b_j^*$, we find that 
\begin{align*}
 a_n^*&=\pi(x)\sum_{j=0}^{n}b_{n-j}^*+\sum_{j=0}^{n}(P_x(X_j=x)-\pi(x))b_{n-j}^*\\
 &=-\pi(x)\sum_{j=n+1}^{\infty}b_j^* + \sum_{j=0}^{n}(P_x(X_j=x)-\pi(x))b_{n-j}^*.
\end{align*}
Hence,
\begin{align*}
\sum_{n=0}^{\infty}|a_n^*|&\leq \pi(x)\sum_{n=0}^{\infty}\sum_{j=n+1}^{\infty}|b_j^*| + \sum_{n=0}^{\infty}\sum_{j=0}^{n}|P_x(X_{n-j}=x)-\pi(x)||b_j^*|\\
&\leq \pi(x)\sum_{j=0}^{\infty}(j+1)|E_xf_c(X_j)I(\tau(x)>j)|+ \sum_{j=0}^{\infty}|b_j^*|\sum_{n=j}^{\infty}|P_z(X_{n-j}=x)-\pi(x)|\\
&\leq \pi(x)E_x\sum_{j=0}^{\infty}(j+1)|f_c(X_j)| I(\tau(x)>j) \\
&\qquad\qquad +E_x\sum_{j=0}^{\infty}(j+1)|f_c(X_j)|I(\tau(x)>j)\cdot \sum_{n=0}^{\infty}|P_x(X_n=x)-\pi(x)|\\
&= \pi(x)E_x \sum_{j=0}^{\tau(x)-1}(j+1)|f_c(X_j)| + E_x\sum_{j=0}^{\tau(x)-1}|f_c(X_j)|\cdot \sum_{n=0}^{\infty}|P_x(X_n=x)-\pi(x)|<\infty,
\end{align*}
provided that
\begin{align}
\sum_{n=0}^{\infty}|P_x(X_n=x)-\pi(x)|<\infty.\label{eq220}
\end{align}
The finiteness of \eqref{eq220} follows from Proposition 11.1.1 and Theorem 13.4.5 of \citet{meynMarkovChainsStochastic2012}. 
 \end{proof}

\begin{remark}
 Observe that when \eqref{eq219} holds for each $x\in S$, it is clear that
 \begin{align*}
 g^*(x)\overset{\Delta}{=} \sum_{j=0}^{\infty}E_xf_c(X_j)
 \end{align*}
 is well-defined, finite-valued, and satisfies \eqref{eq24}. So, when $\eqref{eq214}$ is valid and $E_z\tau(z)^2<\infty$ for some $z\in S$, the infinite sum representation $g^*=(g^*(x):x\in S)$ is an alternative means of representing the solution to Poisson's equation \eqref{eq24}.
\end{remark}

 We turn next to the development of Lyapunov criteria that characterize our weakened sufficient conditions under which solutions of Poisson's equation exist, and exhibit the summability (Theorem \ref{thm4}) and limiting behaviors (Theorem \ref{thm3}) described above. 

 Our results are a simple consequence of the following special case of the Comparison Theorem; see p.344 of \citet{meynMarkovChainsStochastic2012}.

 \begin{proposition} \label{prop2} Suppose that $X=(X_n:n\geq 0)$ is an irreducible and recurrent Markov chain. Assume that there exists $v:S\rightarrow \mathbb R_+$, $w:S\rightarrow \mathbb R_+$, and a finite subset $K\subseteq S$ for which $E_x v(X_1)<\infty$ for $x\in K$ and
 \begin{align}
 E_xv(X_1)\leq v(x)-w(x)\label{eq221}
 \end{align}
 for $x\in K^c$. Then, there exists $c<\infty$ such that
 \begin{align}
 E_x\sum_{j=0}^{\tau(z)-1}w(X_j) \leq v(x)+c\label{eq222}
 \end{align}
 for $x\in S$ and $z\in K$. 
 \end{proposition}
 \begin{proof}  We first note that if we set 
 \begin{align*}
 d= \max_{x\in K} [E_xv(X_1) -v(x)+w(x)],
 \end{align*}
 then \eqref{eq221} implies that 
 \begin{align*}
 E_xv(X_1) \leq v(x)-w(x)+dI(x\in K)
 \end{align*}
 for $x\in S$. We then apply the Comparison Theorem to conclude that
 \begin{align*}
 E_x\sum_{j=0}^{\tau(z)-1} w(X_j) \leq v(x) + dE_x \sum_{j=0}^{\tau(z)-1}I(X_j\in K)
 \end{align*}
 for $x\in S$. If $\Lambda_0=\inf\{n\geq 0:X_n\in K\}$, with $\Lambda_m=\inf\{n>\Lambda_{m-1}:X_n\in K\}$ for $n\geq 1$, the recurrence of $X$ implies that the $\Lambda_m$'s are finite-valued. Furthermore, $(X_{\Lambda_m}:m\geq 0)$ is the Markov chain on $K$ that records the successive visits of $X$ to $K$. Obviously, for $x\in S$, 
 \begin{align*}
 E_x \sum_{j=0}^{\tau(z)-1} I(X_j\in K) \leq E_x\gamma(z)
 \end{align*}
 where $\gamma(z)=\inf\{m\geq 1:X_{\Lambda_m}=z\}$. Since $X$ is irreducible, the finite-state Markov chain $(X_{\Lambda_m}:m\geq 0)$ is irreducible. The finiteness of $|K|$ implies that $E_y \gamma(z)<\infty$ for $y,z\in K$. If we set $c=d\max\{E_y\gamma(z):y,z\in K\}$, we obtain \eqref{eq222}. 
 \end{proof}

 \begin{theorem}\label{thm5} Suppose that $X$ is an irreducible Markov chain, and let $K\subseteq S$ be a finite subset. Consider the following conditions:

 \begin{enumerate}[label=\alph*)]
 \item There exists $v_1:S\rightarrow \mathbb R_+$ such that $(Pv_1)(x)<\infty$ for $x\in K$ and
 \begin{align*}
 (Pv_1)(x)\leq v_1(x)-1
 \end{align*}
 for $x\in K^c$;
 \item There exists $v_2:S\rightarrow \mathbb{R}_+$ such that $(Pv_2)(x)<\infty$ for $x\in K$ and
 \begin{align*}
 (Pv_2)(x)\leq v_2(x)-|f(x)|
 \end{align*}
 for $x\in K^c$;
 \item There exist $v_3:S\rightarrow \mathbb R_+$ and $v_4:S\rightarrow \mathbb R_+$ such that $(Pv_3)(x)+(Pv_4)(x)<\infty$ for $x\in K$,
 \begin{align*}
 (Pv_3)(x)\leq v_3(x)-v_1(x),
 \end{align*}
 and
 \begin{align*}
 (Pv_4)(x)\leq v_4(x)-v_2(x).
 \end{align*}
 \end{enumerate}
 Then:
 \begin{enumerate}[label=\roman*)]
  \item If $a)$ holds, $X$ is positive recurrent and a unique stationary distribution $\pi$ exists.
  \item If $a)$ and $b)$ hold, then Theorems \ref{thm1} and \ref{thm2} are valid.
  \item If $a)$, $b),$ and $c)$ hold, then Theorems \ref{thm1} through \ref{thm4} are valid.
 \end{enumerate} 
 \end{theorem}

 \begin{proof}  Condition $a)$ and Proposition 2 imply that $E_x\tau(z)\leq v_1(x)+c$ for $x\in S$, whereas condition $b)$ implies that for $x\in S$,
 \begin{align}
 E_x\sum_{j=0}^{\tau(z)-1}|f(X_j)| \leq v_2(x)+c\label{eq213-not-used}
 \end{align}
 for some $c\in K$, proving $i)$ and $ii)$ when we set $x=z$. (Of course, these results are well known; see \citet{meynMarkovChainsStochastic2012}.) If $c)$ is also in force, then Proposition 2 implies that 
 \begin{align*}
 E_z\sum_{j=0}^{\tau(z)-1}v_1(X_j) \leq v_3(z)+c
 \end{align*}
 and
 \begin{align*}
 E_z\sum_{j=0}^{\tau(z)-1}v_2(X_j) \leq v_4(z)+c.
 \end{align*}
 Consequently, 
 \begin{align*}
 \infty> E_z \sum_{j=0}^{\tau(z)-1} (\tau(z)-j) = E_z \frac{\tau(z)(\tau(z)+1)}{2}
 \end{align*}
 and
 \begin{align*}
 \infty>E_z \sum_{j=0}^{\tau(z)-1}\sum_{k=j}^{\tau(z)-1}|f(X_k)| = E_z\sum_{k=0}^{\tau(z)-1}(k+1)|f(X_k)|,
 \end{align*}
 for $z\in K$, proving part $iii)$ for $z\in K$. These conditions imply Theorems \ref{thm3} and \ref{thm4}.  
 \end{proof}

 \begin{remark} As usual, these Lyapunov conditions are also necessary conditions, in order that the required moments of Theorems 1 through 4 be finite (for $f$ replacing $f_c$). For example, in order that
 \begin{align}
 E_z \sum_{j=0}^{\tau(z)-1}(j+1)|f(X_j)|<\infty,\label{eq224-new}
 \end{align}
it must be that
\begin{align*}
q(x)\overset{\Delta}{=}E_x\sum_{j=0}^{T_K-1}(j+1)|f(X_j)|<\infty
\end{align*}
for all $x\in S$, where $T_K=\inf\{n\geq 0: X_n\in K\}$ with $K\supseteq\{z\}$. But the function $q=(q(x):x\in S)$ is then a finite-valued solution of
\begin{align*}
q(z)=(Pq)(x)+\tilde q(x)
\end{align*}
for $x\in S$, where $\tilde q=(\tilde q(x):x\in S)$ solves
\begin{align*}
\tilde q(x)=(P\tilde q)(x) + |f(x)|I(x\in K^c)
\end{align*}
for $x\in S$. This, of course, implies that $q$ and $\tilde q$ satisfy the Lyapunov inequalities for $v_2$ and $v_4$ with equality, so that if \eqref{eq224-new} holds, there necessarily exist solutions of the associated Lyapunov inequalities. Similarly, the Lyapunov inequalities for $v_1$ and $v_3$ are necessary for the finiteness of $E_z\tau^2(z)$.
\end{remark}

 Theorem 5 therefore provides (weak) sufficient conditions under which Theorems \ref{thm1} through \ref{thm4} are valid. Assuming that one can find such suitable Lyapunov conditions, one can then validate Theorems \ref{thm1} through \ref{thm4} under close to minimal conditions.

\section{Two Illustrative Examples}\label{sec3}

Our first example illustrates the fact that when $S$ is infinite, there can be infinitely many solutions of Poisson's equation; see, for example, \citet{bhulaiUniquenessSolutionsPoisson2003}. However, the example below shows that when $X$ is irreducible, there can even exist infinitely many linearly independent solutions of Poisson's equation. So, uniqueness can fail badly.

Given the form of (1.1), non-uniqueness emerges from consideration of solutions $h$ to the linear system
\begin{align}
Ph=h.\label{eq31}
\end{align}
A function $h$ satisfying \eqref{eq31} is called a \emph{harmonic function}. Of course, the constant function $h$ always solves $\eqref{eq31}$, and this is the source of the fact that even well-behaved solutions to Poisson's equation are only unique up to an additive constant.

\begin{example}\label{example1} Consider the Markov chain $X$ defined on the state space $S=\{0\}\cup\{(i,j):i\geq 1, j\geq 1\}$ with transition probabilities defined by
\begin{align*}
P(x,y) = \begin{cases} p, & x =(i,j), y=(i+1,j), i,j\geq 1\\
q,& x = (i,j), y=(i-1,j), i\geq 2,j\geq 1\\
q,& x=(1,j), y=0, j\geq 1\\
r_j, & x=0,y=(1,j), j\geq 1\\
1-\sum_{j=1}^\infty r_j, & x=y=0\\
0, & \text{else},
\end{cases}
\end{align*}
where $0<p<1/2,q=1-p,r_j>0$ for $j\geq 1$, and $\sum_{j=1}^\infty r_j<1$. This Markov chain is irreducible, aperiodic, and positive recurrent (since $p<q$). To compute its stationary distribution, we note that $\pi=(\pi(x):x\in S)$ satisfies the linear system of equations
\begin{align}
\pi(i,j) = p \pi(i-1, j)+q\pi(i+1,j)\label{eq32}
\end{align}
for $i\geq 2,j\geq 1$ while
\begin{align}
\pi(1,j) = r_j\pi(0) + q\pi(2,j)\label{eq33}
\end{align}
and
\begin{align}
\pi(0) = (1-\sum_{j=1}^{\infty}r_j)\pi(0) + q\sum_{j=1}^{\infty}\pi(1,j).\label{eq34}
\end{align}
The general solution of \eqref{eq32} is 
\begin{align}
\pi(i,j)=a_j+b_j\left(\frac{p}{q}\right)^i\label{eq35}
\end{align}
for $i,j\geq 1$. Relation \eqref{eq33} implies that
\begin{align}
a_j+b_j = \frac{r_j}{p}\pi(0),
\end{align}
while \eqref{eq34} requires that the $a_j$'s be summable and satisfy
\begin{align}
  \sum_{j=1}^\infty a_j=0.\label{eq37}
\end{align}
Equations \eqref{eq35} and \eqref{eq37} imply that the solution space corresponding to the stationary equations is infinite-dimensional. Of course, \eqref{eq35} implies that $\pi(i,j)\rightarrow a_j$ as $i\rightarrow\infty$. In view of \eqref{eq37}, any solution for which $a_j\neq 0$ for some $j\geq 1$ must therefore have both the property that the $\pi(x)$'s are non-summable and that the $\pi(x)$'s are of mixed-sign. Hence, if $\pi$ is to be a probability distribution, we must have $a_j=0$ for $j\geq 1$, in which case
\begin{align*}
\pi(i,j) = \frac{r_j}{p}\left(\frac{p}{q}\right)^i\pi(0)
\end{align*}
for $i,j\geq 1$, with
\begin{align*}
\pi(0) = (1+\frac{\sum_{j=1}^{\infty}r_j}{q-p})^{-1}.
\end{align*}
The solution space corresponding to the space of harmonic functions is similarly infinite-dimensional. In particular, \eqref{eq31} is given by the linear system
\begin{align*}
ph(i+1,j) + qh(i-1,j) = h(i,j)
\end{align*}
for $i\geq 2,j\geq 1$, with
\begin{align*}
ph(2,j) + qh(0) = h(1,j)
\end{align*}
for $j\geq 1$, and 
\begin{align*}
(1-\sum_{j=1}^{\infty}r_j)h(0) + \sum_{j=1}^{\infty}r_jh(1,j) = h(0).
\end{align*}
The general solution of this linear system is given by
\begin{align}
h(i,j) = h(0) + \tilde b_j[(\frac{q}{p})^i-1]\label{eq38}
\end{align}
for $i,j\geq 1$, where the sequence $(\tilde b_j:j\geq 1)$ must be both summable and satisfy
\begin{align}
\sum_{j=1}^{\infty} \tilde b_j=0.\label{eq39}
\end{align}
\end{example}

An interesting feature of Example 1 is that conditional on $X_0=x\in S$, the birth-death structure of $X$ implies that the random variable (rv) $X_n=(I_n,J_n)$ must be such that $I_n\leq i+n$. As a result, $(h(X_n):n\geq 0)$ is $P_x$-integrable and hence is a $P_x$-martingale for each $x\in S$. Consequently, whenever $f\in L^1(\pi)$, 
\begin{align*}
g_z(X_n) + \sum_{i=0}^{n-1}f_c(X_j) + h(X_n)
\end{align*}
is a $P_x$-martingale for all the harmonic functions characterised by \eqref{eq38} and \eqref{eq39}, and $g_z+h$ is a solution of Poisson's equation.

Note that any non-constant harmonic function fails to be $\pi$-integrable, so $g_z+h$ is not $\pi$-integrable. Hence, this example establishes that even when a solution to Poisson's equation induces a martingale, that solution may be badly behaved (e.g. it may not be $\pi$-integrable).

This raises the question of whether there is an alternative martingale property that characterizes well-behaved solutions to Poisson's equation (e.g. solutions of the form $g_z+ce$ for $c\in \mathbb R$, where $e(x)=1$ for $x\in S$). Our next result provides one such characterization; see part b.).

\begin{theorem}\label{thm6} Assume that $f\in L^1(\pi)$. 
\begin{enumerate}[label=\alph*)]
\item For $x,y,z\in S$,
\begin{align}
(g_y(X_{\tau(z)\land n})+\smashoperator{\sum_{j=0}^{(\tau(z)\land n)-1}}f_c(X_j):n\geq 0)\label{eq310}
\end{align}
is a $P_x$-uniformly integrable martingale adapted to $(\mathcal{F}_n:n\geq 0)$.
\item Suppose that there exists $z\in S$ and $g:S\rightarrow \mathbb{R}$ such that for each $x\in S$,
\begin{align*}
(g(X_{\tau(z)\land n}) + \smashoperator{\sum_{j=0}^{(\tau(z)\land n)-1}}f_c(X_j):n\geq 0)
\end{align*}
is a $P_x$-uniformly integrable martingale adapted to $(\mathcal{F}_n:n\geq 0)$. Then, for each $y\in S$, $g(x)=g_y(x)+g(y)e(x)$.
\end{enumerate}
\end{theorem}

\begin{proof} For part a), we first recognize that $g_y(\cdot)$ and $g_z(\cdot)$ differ by an additive constant (by Theorem \ref{thm1}), so it is sufficient to prove the result for $y=z$. We further note that Theorem \ref{thm2} establishes \eqref{eq28} is a $P_x$-martingale, and hence optional sampling implies that \eqref{eq310} is a $P_x$-martingale. So, it remains only to prove that \eqref{eq310} is $P_x$-uniformly integrable. Note that
\begin{align*}
g_z(X_{\tau(z)\land n})+\sum_{j=0}^{(\tau(z)\land n)-1}f_c(X_j)
\overset{a.s.}{\rightarrow} g_z(X(\tau(z))) + \sum_{j=0}^{\tau(z)-1}f_c(X_j) = \sum_{j=0}^{\tau(z)-1}f_c(X_j)
\end{align*}
as $n\rightarrow\infty$. In view of Theorem 4.6.3 of \citet{durrettProbabilityTheoryExamples2019}, it suffices to prove that
\begin{align}
E_x|g_z(X_{\tau(z)\land n}) +\smashoperator{\sum_{j=0}^{(\tau(z)\land n)-1}}|f_c(X_j)| \rightarrow E_x|\sum_{j=0}^{\tau(z)-1}f_c(X_j)|\label{eq311}
\end{align}
as $n\rightarrow\infty$. But 
\begin{align}
|g_z(X_{\tau(z)\land n}) + \smashoperator{\sum_{j=0}^{(\tau(z)\land n)-1}}f_c(X_j)| &\leq  |g_z(X_n)|I(\tau(z)>n) + \sum_{j=0}^{\tau(z)-1}|f(X_j)| + |\pi f|\tau(z)\label{eq312}
\end{align}
and the latter two rv's on the right-hand size of \eqref{eq312} have finite $P_x$-expectation because $f\in L^1(\pi)$ and $X$ is positive recurrent. For the first term, note that
\begin{align*}
E_x(|g_z(X_n)| I(\tau(z)>n))&= E_x |\sum_{j=n}^{\beta_n(z)-1}f_c(X_j)| I(\tau(z)>n)\\
&\leq E_x\sum_{j=n}^{\beta_n(z)-1}|f_c(X_j)| I(\tau(z)>n)\\
&=E_z \sum_{j=n}^{\tau(z)-1}|f_c(X_j)|I(\tau(z)>n)\rightarrow 0
\end{align*}
as $n\rightarrow\infty$. As a consequence, \eqref{eq311} follows and we have established the required uniform integrability. 

For part $b)$, note that the uniform integrability and martingale property imply that
\begin{align*}
g(x) = E_xg(X_{\tau(z)\land n}) + \smashoperator{\sum_{j=0}^{(\tau(x)\land n)-1}} f_c(X_j)
\rightarrow g(z) + E_x\sum_{j=0}^{\tau(z)-1}f_c(X_j) = g(z) + g_z(x)
\end{align*}
as $n\rightarrow\infty$, proving the result for $y=z$. The case of general $y$ is handled by just noting that $g_z(\cdot)$ and $g_y(\cdot)$ differ by an additive constant; see Theorem \ref{thm1}.  \\ 
\end{proof} 

While our first example illustrates non-uniqueness issues related to Poisson's equation, our second example provides further insight into existence and representation issues related to Poisson's equation. Theorem \ref{thm1} shows that whenever $f\in L^1(\pi)$, $g_z$ solves Poisson's equation. However, a commonly used representation of the solution of Poisson's equation is that given by \eqref{eq221}, namely
\begin{align}
g(x) = \sum_{j=0}^{\infty}E_xf_c(X_j).\label{eq313}
\end{align}
Our example below shows that there are Markov chains $X$ and functions $f\in L^1(\pi)$ for which the solution $g_z$ is well-defined, while the representation \eqref{eq313} is not well-defined (since the sum fails to be summable). Hence, $g_z$ is a universal representation of the solution to Poisson's equation, while the ``potential-theoretic'' representation \eqref{eq313} requires more regularity in order that it be valid.

\begin{example} \label{example2} Let $\beta_1,\beta_2,...$ be a sequence of independent and identically distributed (iid) positive integer-valued rv's, and put
\begin{align*}
S_n=\beta_1+\beta_2+...+\beta_n,
\end{align*}
with $S_0=0$. Let $\ell(n)=\max\{j:S_j\leq n\}$ and let
\begin{align*}
X_n = n-S_{\ell(n)}
\end{align*}
be the ``current age'' Markov chain associated with the inter-renewal times $\beta_1,\beta_2,...$ . Put $f(x)=\delta_{x0}$, so that $f(\cdot)$ is 1 when $x=0$ and 0 otherwise. Then,
\begin{align*}
E_0f(X_n) = P_0(X_n=0) \overset{\Delta}{=}u_n,
\end{align*}
where $(u_n:n\geq 0)$ is the renewal sequence associated with the increment probability mass function $(p_j:j\geq 1)$ given by $p_j=P(\beta_1=j)$ for $j\geq 1$. Assume that $(p_j:j\geq 1)$ is a positive sequence for which there exists $c>0$ and $\alpha>1$ for which
\begin{align*}
P(\beta_1>n)\sim cn^{-\alpha}
\end{align*}
as $n\rightarrow\infty$ (where we write $a_n\sim b_n$ as $n\rightarrow\infty$ when $a_n/b_n\rightarrow 1$ as $n\rightarrow\infty$). Then, $E\beta_1<\infty$ and $X$ is positive recurrent with $\pi(0)=\lambda\overset{\Delta}{=}1/E\beta_1$.

According to Lemma 4 of \citet{frenkBehaviorRenewalSequence1982},
\begin{align*}
u_n-\lambda \sim \frac{\lambda^2n P(\beta_1>n)}{\alpha-1}
\end{align*}
as $n\rightarrow\infty$, from which it follows that
\begin{align*}
P_0(X_n=0) - \pi(0) \sim\frac{\lambda^2 c n^{1-\alpha}}{(\alpha-1)}
\end{align*}
as $n\rightarrow\infty$. As a consequence, if $\alpha\in (1,2]$, \eqref{eq313} fails to be summable at $x=0$, providing the required example. 
\end{example}

\section{The CLT and LIL for Markov Chains}\label{sec4}

We finish this paper by developing a weakened Lyapunov criterion for the validity of the central limit theorem and law of the iterated logarithm for countable state Markov chains. It is well known that if there exists $z\in S$ such that
\begin{align}
E_z\left(\sum_{j=0}^{\tau(z)-1}f_c(X_j)\right)^2<\infty,\label{eq41}
\end{align}
then, regardless of the initial distribution for $X_0$,
\begin{align}
\frac{S_n(f)-n\pi f}{\sqrt{n}}\Rightarrow \sigma N(0,1)\label{eq42}
\end{align}
as $n\rightarrow\infty$, where $\Rightarrow$ denotes weak convergence, $N(0,1)$ is a standard normal rv with mean 0 and variance 1, and
\begin{align}
\sigma^2 \overset{\Delta}{=} \frac{E_z\left(\sum_{j=0}^{\tau(z)-1}f_c(X_j)\right)^2}{E_z\tau(z)},\label{eq43}
\end{align}
see p.99 of \citet{chungMarkovChains1967} or \citet{glynnLimitTheoremsCumulative1993}. Furthermore, Theorem 14.4 of \citet{chungMarkovChains1967} proves that if $\eqref{eq41}$ is valid for one $z\in S$, then it is valid for all $z\in S$, and the right-hand side of \eqref{eq43} does not depend on $z$. In \citet{glynnNecessaryConditionsLimit2002}, it is shown that \eqref{eq42} implies \eqref{eq41}, so that \eqref{eq41} is a necessary and sufficient condition for the CLT \eqref{eq42}, in the presence of positive recurrence and $f\in L^1(\pi)$.

We now slightly strengthen \eqref{eq41} to the condition
\begin{align}
E_z(\sum_{j=0}^{\tau(z)-1}|f_c(X_j)|)^2<\infty.\label{eq44}
\end{align}

\begin{theorem}\label{thm7} Suppose that $X$ is irreducible, and let $K\subseteq S$ be a finite subset. Suppose that there exists $v_1\in S\rightarrow \mathbb R_+$ and $v_2\in S\rightarrow\mathbb R_+$ such that $(Pv_i)(x)<\infty$ for $x\in K$ and $i=1,2$ and
\begin{align}
(Pv_1)(x)&\leq v_1(x)-(|f(x)|+1)\label{eq45}\\
(Pv_2)(x)&\leq v_2(x)-(|f(x)|+1)v_1(x)\label{eq46}
\end{align}
for $x\in K^c$. Then, $X$ is positive recurrent, $f\in L^1(\pi)$, and $\eqref{eq44}$ is valid, so that the CLT \eqref{eq42} holds. Furthermore, $\sigma^2$ can then be expressed as \eqref{eq43} or, equivalently, as 
\begin{align}
\sigma^2 = 2E_\pi g_z(X_0)f_c(X_0)- E_\pi f_c(X_0)^2.\label{eq47}
\end{align}
\end{theorem}

\begin{proof} Proposition \ref{prop2} and \eqref{eq45} imply that $X$ is positive recurrent with $f\in L^1(\pi)$, and
\begin{align*}
E_x (\tau(z) + \sum_{j=0}^{\tau(z)-1}|f(X_j)|) \leq v_1(x)+c_1
\end{align*}
for $x\in S,z\in K,$ and some $c_1\in\mathbb R_+$. On the other hand, Proposition \ref{prop2} and \eqref{eq46} imply that there exists $c_2\in \mathbb R_+$ for which
\begin{align*}
E_x\sum_{j=0}^{\tau(z)-1}|f(X_j)+1|v_1(X_j)  \leq v_2(x)+c_2.
\end{align*}
Consequently, 
\begin{align*}
E_x(\sum_{j=0}^{\tau(z)-1} |f(X_j)+1|)^2&\leq 2 E_x\sum_{j=0}^{\tau(z)-1}|f(X_j)+1|\sum_{k=j}^{\tau(z)-1}(|f(X_k)|+1)\\
&\leq 2E_x \sum_{j=0}^{\tau(z)-1}(|f(X_j)+1|)(v_1(X_j)+c_1)\\
&\leq 2(v_2(x)+c_2) + c_1(v_1(x)+c_1)<\infty
\end{align*}
for $x\in S$, proving \eqref{eq44}.

To prove \eqref{eq47}, note that in the presence of \eqref{eq44}, we see that
\begin{align*}
E_z(\sum_{j=0}^{\tau(z)-1}f_c(X_j))^2 &= 2E_z\sum_{j=0}^{\tau(z)-1}f_c(X_j)\sum_{k=j}^{\tau(z)-1}f_c(X_k)- E_z\sum_{j=0}^{\tau(z)-1}f_c(X_j)^2\\
&=2E_z \sum_{j=0}^{\tau(z)-1}f_c(X_j) g_z(X_j) - E_z\sum_{j=0}^{\tau(z)-1}f_c^2(X_j)\\
&= E_z \tau(z)(2E_\pi f_c(X_0)g_z(X_0)-E_\pi f_c(X_0)^2),
\end{align*}
establishing the result. 
\end{proof}

\begin{remark} This Lyapunov criterion for Markov chain CLT weakens the existing criterion of \citet{glynnLiapounovBoundSolutions1996}. In particular, when one writes $S_n(f)$ in terms of the martingale \eqref{eq28}, one is naturally led to consideration of the associated martingale differences given by
\begin{align*}
D_i = g_z(X_i) - (Pg_z)(X_{i-1})
\end{align*}
for $i\geq 1$. In order that $D_i\in L^2(\pi)$, it seems appropriate to demand that $g_z\in L^2(\pi)$; see Theorem 4.1 of \citet{glynnLiapounovBoundSolutions1996}. This is also the condition used in earlier work by \citet{maigretTheoremeLimiteCentrale1978} on the Markov chain CLT. However,  the expression \eqref{eq47} makes clear that the key requirement in the Markov chain CLT is $g_zf_c\in L^1(\pi)$. This is effectively what \eqref{eq46} is verifying.

To see that use of \eqref{eq46} gives better conditions than does \citet{glynnLiapounovBoundSolutions1996}, consider the Markov chain defined by
\begin{align}
X_{n+1}=[X_n+Z_{n+1}]^+\label{eq48}
\end{align}
where $[x]^+=\max(x,0)$ and the $Z_i$'s are independent and identically distributed (iid) integer-valued rv's with $-\infty<EZ_1<0$. As is well-known, this Markov chain arises naturally in the modeling of queues. Use of Theorem 4.1 in that paper for $f(x)=x$ leads to the requirement that $E|Z_1|^5<\infty$. On the other hand, our Theorem \ref{thm7} requires only that $EZ_1^4<\infty$, thereby weakening the moment requirement on $Z_1$. (Use $v_1(x)=ax^2$ and $v_2(x)=a'x^4$ for suitably chosen $a$ and $a'$.)

We now argue that $EZ_1^4<\infty$ is the natural condition that arises in connection to the CLT for the Markov chain defined by \eqref{eq48}. We first note that $X_n$ is a non-decreasing function of the independent rv's $X_0,Z_1,...,Z_n$. Consequently, the $X_n$'s are associated rv's; see \citet{barlowStatisticalTheoryReliability1975}, p.29-31. Hence,
\begin{align*}
E_\pi(X_0\land r)(X_n\land r) \geq E_\pi X_0\land r \cdot E_\pi X_n\land r
\end{align*}
for $r\geq 0$. Sending $r\rightarrow\infty$, we conclude that
\begin{align*}
E_\pi X_0X_n \geq E_\pi X_0\cdot E_\pi X_n
\end{align*}
via the Monotone Convergence Theorem. If the CLT holds, this implies the existence of a finite non-negative $\alpha\in \mathbb R$ such that $n^{-1}\sum_{j=0}^{n-1}X_j\overset{p}{\rightarrow}\alpha$ as $n\rightarrow\infty$. Because the $X_j$'s are non-negative, $\alpha$ must equal $E_\pi X_0$. So, $E_\pi X_0<\infty$ and
\begin{align*}
c_n \overset{\Delta}{=}E_\pi X_0X_n - E_\pi X_0\cdot E_\pi X_n \geq 0
\end{align*}
for $n\geq 0$. Another truncation and monotone convergence argument establishes that 
\begin{align}
\frac{1}{n}E_\pi (\sum_{j=0}^{n-1}(X_j-E_\pi X_0))^2 = c_0 + 2\sum_{j=1}^{n}\frac{(n-j)}{n}c_j.\label{eq49}
\end{align}
Since
\begin{align*}
\frac{1}{2}\sum_{j=0}^{\lfloor n/2\rfloor}c_j \leq c_0 + 2 \sum_{j=1}^{n}\frac{(n-j)}{n}c_j \leq c_0+ 2\sum_{j=1}^{\infty}c_j,
\end{align*}
it is evident that the left-hand side of \eqref{eq49} has a finite limit if and only if
\begin{align}
c_0 + 2\sum_{j=1}^{\infty}c_j <\infty.\label{eq410}
\end{align}
According to \citet{daleySerialCorrelationCoefficients1968}, \eqref{eq410} holds (in the presence of a blanket assumption that $E|Z_1|^3<\infty$) if and only if $EZ_1^4<\infty$. Of course, demanding that the left-hand side of \eqref{eq49} have a limit follows whenever the CLT \eqref{eq42} is in force and the square of the left-hand side of \eqref{eq42} is $P_\pi$-uniformly integrable.

\begin{remark} Another indication that $EZ_1^4$ is a necessary condition for the CLT in this setting is that when one allows the $Z_i$'s to be continuous rv's, \citet{glynnPoissonEquationRecurrent1994} explicitly solves the associated Poisson's equation for the $M/G/1$ queue. In that setting, the solution to Poisson's equation for $f$ is quadratic, so that $g_zf_c$ is cubic. As a result, it is well known that $EZ_1^4$ is necessary in order that $g_zf_c$ be $\pi$-integrable (\citet{asmussenAppliedProbabilityQueues2008}), and hence that $\sigma^2$ as given by \eqref{eq47} be finite.

\end{remark}

We now turn to the law of the iterated logarithm for $S_n(f)$. The following is an immediate consequence of Theorem \ref{thm7} above and Theorem 5, p.106 of \citet{chungMarkovChains1967}.

\begin{theorem}\label{thm8} Under the same hypotheses as in Theorem \ref{thm7},
\begin{align*}
\varlimsup_{n\rightarrow\infty}\left|\frac{S_n(f)-n\pi f}{\sqrt{n\log\log n}}\right|=|\sigma|\qquad P_x\  \text{a.s.}
\end{align*}
for each $x\in S$, where $\sigma^2$ is as given in Theorem \ref{thm7}.
\end{theorem}

\citet{glynnLimitTheoremsCumulative1993} provide necessary and sufficient conditions for the LIL when the process is regenerative. The condition is stronger than \eqref{eq41}, but slightly weaker than \eqref{eq42}. So, Theorem \ref{thm8} gives close to optimal conditions for the LIL.
\end{remark}

\bibliographystyle{newapa-and}
\bibliography{Markov-static}

\end{document}